\documentclass{amsart}
\usepackage{amssymb}
\usepackage{amsfonts}

\newtheorem{thm}{Theorem}[section]

\theoremstyle{plain}

\newtheorem{cor}[thm]{Corollary}

\newtheorem{definition}[thm]{Definition}

\newtheorem{remk}[thm]{Remark}

\newtheorem{exa}[thm]{Example}

\numberwithin{equation}{section}

\begin{document}

\title {Square-difference factor absorbing ideals of a commutative ring}

\author{David F. Anderson, Ayman Badawi, and Jim Coykendall}

\date{\today}



\begin{abstract}
Let $R$ be a commutative ring with $1 \neq 0$. A proper ideal $I$ of $R$ is a {\it square-difference factor absorbing ideal} (sdf-absorbing ideal) of $R$ if whenever  $a^2 - b^2 \in I$ for $0 \neq a, b \in R$, then $a + b \in I$ or  $a - b \in I$. In this paper, we introduce and investigate sdf-absorbing ideals.
\end{abstract}

 \maketitle

\section{Introduction} \label{s1}
In this paper, we introduce and study square-difference factor absorbing  ideals of a commutative ring $R$ with nonzero identity, where a proper ideal $I$ of $R$ is a {\it square-difference factor absorbing ideal} (sdf-absorbing ideal) of $R$ if whenever  $a^2 - b^2 \in I$ for $0 \neq a, b \in R$, then $a + b \in I$ or  $a - b \in I$. A prime ideal is an sdf-absorbing ideal (but not conversely). In particular, $\{0\}$ is an sdf-absorbing ideal in any integral domain. For other generalizations of prime ideals, see \cite{AB, AB2, B}. We also introduce and briefly study weakly square-difference factor absorbing ideals, where a proper ideal $I$ of $R$ is a {\it weakly square-difference factor absorbing ideal} (weakly sdf-absorbing ideal) of $R$ if whenever  $0 \neq a^2 - b^2 \in I$ for $0 \neq a, b \in R$, then $a + b \in I$ or  $a - b \in I$

In Section~\ref{s2}, we give some basic properties of sdf-absorbing ideals. We show that a nonzero sdf-absorbing ideal is a radical ideal (Theorem~\ref{t1}), and the converse holds when char$(R) = 2$  (Theorem~\ref{t2}).  Moreover, a nonzero sdf-absorbing ideal is a prime ideal when $2 \in U(R)$ (Theorem~\ref{t3}). Throughout this paper, properties of $2 \in R$ will play an important role. In Section~\ref{s3}, we study when all proper ideals or nonzero proper ideals of a commutative ring are sdf-absorbing ideals. In particular, every nonzero proper ideal of a commutative ring $R$ is an sdf-absorbing ideal of $R$ if and only if $R/nil(R)$ is a von Neumann regular ring (Theorem~\ref{t20}). In addition, we determine when every proper ideal or nonzero proper ideal of a commutative von Neumann regular ring is an sdf-absorbing ideal (Theorems~\ref{t5}, \ref{t85}, and \ref{t101}). In Section~\ref{s4}, we give several additional results about sdf-absorbing ideals. For example, we determine the sdf-absorbing ideals in a PID (Corollary~\ref{c5}) and the direct product of two commutative rings (Theorem~\ref{t600}).  We also study sdf-absorbing ideals in polynomial rings (Theorems~\ref{t205} and \ref{t569}), idealizations (Theorem~\ref{t907}), amalgamation rings (Theorem~\ref{t709}), and $D + M$ constructions (Theorem~\ref{t763}). In Section~\ref{s5}, the final section, we  briefly study weakly sdf-absorbing ideals. Several results for sdf-absorbing ideals have analogs for weakly sdf-absorbing ideals. Many examples are given throughout to illustrate the results.

We assume throughout that all rings are commutative with nonzero identity and that $f(1) = 1$ for all ring homomorphisms $f : R \longrightarrow T$. Let $R$ be a commutative ing. Then dim$(R)$ denotes the Krull dimension of $R$, char$(R)$ the characteristic of $R$, $J(R)$ the Jacobson radical of $R$, $nil(R)$ the ideal of nilpotent elements of $R$, $Z(R)$ the set of zero-divisors of $R$, and $U(R)$ the group of units of $R$. The ring $R$ is {\it reduced} if $nil(R) = \{0\}$. Recall that a commutative ring $R$ is {\it von Neumann regular} if for every $x \in R$, there is a $y \in R$ such that $x^2y = x$. Equivalently, $R$ is von Neumann regular if and only if $R$ is reduced and dim$(R) = 0$ (\cite[Theorem 3.1]{H}).

As usual, $\mathbb{Z}$, $\mathbb{Q}$, $\mathbb{Z}_n$, and $\mathbb{F}_q$ will denote the integers, rationals,  integers modulo $n$, and the finite field with $q$ elements, respectively. For any undefined concepts or terminology, see \cite{G, H, K}.

\section{Properties of sdf-absorbing ideals} \label{s2}

In this section, we give some basic properties of square-difference factor absorbing ideals. We begin with the definition.

\begin{definition}
{\rm A proper ideal $I$ of a commutative ring $R$ is a {\it square-difference factor absorbing ideal} (sdf-absorbing ideal) of $R$ if whenever  $a^2 - b^2 \in I$ for $0 \neq a, b \in R$, then $a + b \in I$ or  $a - b \in I$.}
\end{definition}

Clearly a prime ideal is an sdf-absorbing ideal. Although the converse may fail (see Example~\ref{e1}(a)), we next show that nonzero sdf-absorbing ideals are always radical ideals.

\begin{thm}\label{t1}
Let $I$ be a nonzero sdf-absorbing ideal of a commutative ring $R$. Then $I$ is a radical ideal of $R$.
\end{thm}

\begin{proof}
It is sufficient to show that for  $0 \neq a \in R$, we have $a \in I$ whenever $a^2 \in I$. Since $I$ is a nonzero ideal of $R$, there is a $0 \neq i \in I$. Thus $a^2 - i^2 \in I$;  so $a + i \in I$ or $a - i \in I$ since $I$ is an sdf-absorbing ideal of $R$. Hence $a \in I$, and thus $I$ is a radical ideal of $R$.
\end{proof}

\begin{remk}\label{e0}
{\rm (a)  The ``nonzero" hypothesis is needed in Theorem~\ref{t1}. It is easily verified that $\{0\}$ is an sdf-absorbing ideal of $\mathbb{Z}_4$ (cf. Theorem~\ref{t9}), but not a radical ideal of $\mathbb{Z}_4$.
	
(b)  Theorem~\ref{t1} also shows that the ``$a, b \neq 0$" hypothesis is not needed in the definition of sdf-absorbing ideal when the ideal is nonzero.}
\end{remk}

A radical ideal need not be an sdf-absorbing ideal, see Example~\ref{e1}(a). However, the converse of Theorem~\ref{t1} does hold when char$(R) = 2$, and the ``nonzero ideal''  hypothesis is not needed.

\begin{thm}\label{t2}
Let $I$ be a radical ideal of a commutative ring $R$ with char$(R) = 2$. Then $I$ is an sdf-absorbing ideal of $R$.
\end{thm}

\begin{proof}
Let $I$ be a radical ideal of $R$ and $a^2 - b^2 \in I$ for $0 \neq a, b \in R$. Since char$(R) = 2$, we have $(a + b)^2 = a^2 + b^2  = a^2 - b^2\in I$, and thus $a +b \in I$ since $I$ is a radical ideal of $R$. Hence $I$ is an sdf-absorbing ideal of $R$.
\end{proof}

If char$(R) = 2$, then $a - b = a +b$. The next result determines when $a^2 - b^2 \in I$ for $I$ an sdf-absorbing ideal of $R$ implies both $a +b, a - b \in I$.

\begin{thm}\label{l1}
Let $I$ be an sdf-absorbing ideal of a commutative ring $R$. Then the following statements are equivalent.

	{\rm (a)} If $a^2 - b^2 \in I$ for  $0 \neq a, b \in R$, then $a + b, a - b \in I$.

	{\rm (b)} $2 \in I$.

         {\rm (c)} char$(R/I) = 2$.
\end{thm}

\begin{proof}
	$(a) \Rightarrow (b)$ Let $a = b = 1$. Then $a^2 - b^2 = 0 \in I$, and thus  $2 = 1 + 1 = a + b \in I$ by hypothesis.
	
	$(b) \Rightarrow (a)$ Assume that $a^2 - b^2 \in I$ for $0 \neq a, b \in R$.   Then $a + b \in I$ or $a - b \in I$ since $I$ is an sdf-absorbing ideal of $R$, and   $2b \in I$ since $2 \in I$. Thus $a - b = (a + b) -2b \in I$ if $a + b \in I$, and  $a + b = (a - b)+ 2b \in I$ if $ a - b \in I$.

$(b) \Leftrightarrow (c)$ This is clear.
\end{proof}

The next result gives a case where sdf-absorbing ideals are prime ideals. It is easily verified that $\{0\}$ is an sdf-absorbing ideal of $\mathbb{Z}_9$ (cf. Theorem~\ref{t9}); so the ``nonzero'' hypothesis is needed in the following theorem.

\begin{thm}\label{t3}
Let $I$ be a nonzero sdf-absorbing ideal of a commutative ring $R$ with $2 \in U(R)$. Then $I$ is a prime ideal of $R$.
\end{thm}

\begin{proof}
Let $I$ be a nonzero sdf-absorbing ideal of $R$ and $xy \in I$ for $x, y \in R$.  We may assume that  $x, y \not = 0$.  First, assume that $y \not = x$ and $y \not = -x$. Let $a = (x + y)/2, b = (x - y)/2 \in R$. Since $y \not = x$ and $y \not = -x$, we have $a^2 - b^2 = xy \in I$ and $a, b \not = 0$. Thus $x = a + b  \in I$ or  $y = a - b  \in I$ since $I$ is an sdf-absorbing ideal of $R$. Next, assume that $y = x$ or $y = -x$. Then $x^2 \in I$, and hence $x \in I$ since $I$ is a  radical ideal of $R$ by Theorem~\ref{t1}.  Thus $I$ is a prime ideal of $R$.
\end{proof}

The following criterion for an ideal to be an sdf-absorbing ideal will often prove useful.

\begin{thm} \label{t4}
Let $I$ be a proper ideal of a commutative ring $R$. Then the following statements are equivalent.
	
{\rm (a)} $I$ is an sdf-absorbing ideal of $R$.

{\rm (b)} If $ab \in I$ for $a, b \in R \setminus I$, then the system of linear equations $X + Y = a$, $X - Y = b$ has no nonzero solution in $R$ {\rm (}i.e., there are no $0 \neq x, y \in R$ that satisfy both equations{\rm )}.
	
\end{thm}

\begin{proof} $(a) \Rightarrow (b)$ Suppose that $I$ is an sdf-absorbing ideal of $R$, $ab \in I$ for $a, b \in R \setminus I$, and the system of linear equations $X + Y = a$, $X - Y = b$ has a solution in $R$ for some $0 \neq x, y \in R$. Then $x^2 - y^2 = ab \in I$, but $x + y = a \not \in I$ and $x - y = b \not \in I$, a contradiction.
	
	$(b) \Rightarrow (a)$ Assume that $x^2 - y^2 \in I$ for $0 \neq x, y \in R$. Let $a = x + y$ and $b = x - y$. Then $ab = x^2 - y^2  \in I$, and the system of linear equations $X + Y = a, X - Y = b$ has a solution in $R$ for $0 \neq x, y \in R$. Thus $x + y = a \in I$ or $x - y = b \in I$, and hence $I$ is an sdf-absorbing ideal of $R$.
\end{proof}

We next give several examples of sdf-absorbing ideals.

\begin{exa} \label{e1}
{\rm (a)  Using Theorem~\ref{t4}, one can easily verify that a proper ideal $I$ of $\mathbb{Z}$ is an sdf-absorbing ideal of $\mathbb{Z}$  if and only if $I$ is a prime ideal of $\mathbb{Z}$ or $I = 2q\mathbb{Z}$ for some odd prime integer $q$. Note that if $p, q$ are nonassociate odd prime integers, then $pq\mathbb{Z}$ is a radical ideal of $\mathbb{Z}$ which is not an sdf-absorbing ideal of $\mathbb{Z}$. Thus the converse of Theorem~\ref{t1} may fail. See Corollary~\ref{c5} for the general PID case.
	
	(b) Let $R$ be a boolean ring. Then every proper ideal of $R$ is an sdf-absorbing ideal of $R$ since $x^2 = x$ for every $x \in R$.
	
	(c) Let $R = \mathbb{Z}[W, T]$. Since $(2W)(2T) \in WTR$  and the system of linear equations $x + y = 2W$, $x - y = 2T$ has a nonzero solution in $R$, we have that $WTR$ is not an sdf-absorbing ideal of $R$ by Theorem \ref{t4}. However, note that $WTR$ is a radical ideal of $R$.
	
	(d)  Let  $R = \mathbb{Z}_2 \times \cdots \times \mathbb{Z}_2 \times \mathbb{Z}$ and $I = I_1 \times \cdots \times I_n \times J$ be an ideal of $R$. Using Theorem~\ref{t4} again, one can easily verify that $I$ is an sdf-absorbing ideal of $R$ if and only if $J$ is a prime ideal of $\mathbb{Z}$ or $J = 2q\mathbb{Z}$ for some odd prime integer $q$.

   (e)  Let $R = K[X]$, where $K$ is a field, and $I = (X + 1)(X - 1)R$. Then $I$ is a radical ideal of $R$ if and only if char$(K) \neq 2$. Thus $I$ is never an sdf-absorbing ideal of $R$ by Theorem~\ref{t1} and Theorem~\ref{t3}.

     (f)  Let $R = K[X]$, where $K$ is a field, and $I = f_1^{n_1} \cdots k_k^{n_k}R$ for nonassociate irreducible $f_1, \ldots, f_k \in R$ and positive integers $n_1, \ldots, n_k$. Then $I$ is a radical (resp., prime) ideal of $R$ if and only if $n_1 = \cdots = n_k = 1$ (resp., $k = n_k = 1$). If char$(K) = 2$, then $I$ is an sdf-absorbing ideal of $R$ if and only if $I$ is a radical ideal of $R$ by Theorem~\ref{t2}. If char$(K) \neq 2$, then $I$ is an sdf-absorbing ideal of $R$ if and only if $I$ is a prime ideal of $R$ by Theorem~\ref{t3}.

(g)  Let $R$ be a valuation domain. Then every radical ideal of $R$ is a prime ideal (\cite[Theorem 17.1(2)]{G}); so the prime ideals are the only      sdf-absorbing ideals of $R$ by Theorem~\ref{t1}.}
\end{exa}

The next two theorems and corollary follow directly from the definitions and Remark~\ref{e0}(b); so their proofs are omitted.

\begin{thm} \label{t349}
Let $I$ be an sdf-absorbing ideal of a commutative ring $R$, and let $S$ be a multiplicatively closed subset of $R$ with $I \cap S = \emptyset$. Then $I_S$ is an sdf-absorbing ideal of $R_S$.
\end{thm}

\begin{thm}\label{t7}
Let $f : R \longrightarrow T$ be a homomorphism of commutative rings.

{\rm (a)} If  $J$ is a nonzero sdf-absorbing ideal of $T$, then $f^{-1}(J)$ is an sdf-absorbing ideal of $R$.

{\rm (b)} If $f$ is injective and $J$ is an sdf-absorbing ideal of $T$, then $f^{-1}(J)$ is an sdf-absorbing ideal of $R$.

{\rm (c)} If $f$  is surjective and $I$ is an sdf-absorbing ideal of $R$ containing $ker(f)$, then  $f(I)$ is an sdf-absorbing ideal of $T$.
\end{thm}

\begin{cor} \label{c2}
{\rm (a)} Let $R \subseteq T$ be an extension of commutative rings and $J$ an sdf-absorbing ideal of $T$. Then     $J \cap R$ is an sdf-absorbing ideal of $R$.

{\rm (b)}  Let $ J \subseteq I$ be ideals of a commutative ring R. If $I$ is an sdf-absorbing ideal of $R$, then $I/J$ is an sdf-absorbing ideal of $R/J$.

{\rm (c)} If $J \subsetneq I$, then $I/J$ is an sdf-absorbing ideal of $R/J$ if and only if $I$ is an sdf-absorbing ideal of $R$.
\end{cor}

The following examples show that the ``nonzero'' hypothesis is needed in Theorem~\ref{t7}(a) and the ``$ker(f) \subseteq I$'' hypothesis is needed in Theorem~\ref{t7}(c).
	
\begin{exa}\label{e2}
{\rm (a) Let $f: \mathbb{Z} \longrightarrow \mathbb{Z}/4\mathbb{Z} = \mathbb{Z}_4$ be the natural epimorphism. By Remark ~\ref{e0}(a), $\{0\}$ is an sdf-absorbing ideal of $\mathbb{Z}_4$, but $f^{-1}(\{0\}) =  4\mathbb{Z}$ is not an sdf-absorbing ideal of $\mathbb{Z}$ by Example~\ref{e1}(a). Thus the ``nonzero'' hypothesis is needed in Theorem~\ref{t7}(a).

(b)  Let $f : \mathbb{Z}[X] \longrightarrow \mathbb{Z}$ be the epimorphism given by $f(g(X)) = g(0)$. Then $I = (X + 4)$ is a prime ideal, and thus an sdf-absorbing ideal, of $\mathbb{Z}[X]$, but  $f((X+4)) = 4\mathbb{Z}$ is not an sdf-absorbing ideal of $\mathbb{Z}$ by Example~\ref{e1}(a). Note that $ker(f) = (X) \not \subseteq (X + 4) = I$; so the ``$ker(f) \subseteq I$'' hypothesis is needed in Theorem~\ref{t7}(c).}
\end{exa}

\section{When every ideal is an sdf-absorbing ideal} \label{s3}

In this section, we consider when every proper ideal or nonzero proper ideal of a commutative ring is an sdf-absorbing ideal. Recall that every proper ideal of a commutative ring $R$ is a radical ideal if and only if $R$ is von Neumann regular (cf. \cite[Proposition 1.1]{AB1}). We use this fact to show that if every nonzero proper ideal of $R$ is an sdf-absorbing ideal of $R$, then  $R/nil(R)$ is von Neumann regular. In particular, if in addition $R$ is reduced, then $R$ is von Neumann regular.

\begin{thm} \label{t20}  Let $R$ be commutative ring such that every nonzero proper ideal of $R$ is an sdf-absorbing ideal of $R$. Then  $R/nil(R)$ is von Neumann regular. In particular, dim$(R) = 0$. Moreover, if $R$ is not reduced, then $nil(R)$ is the unique minimal nonzero ideal of $R$.
\end{thm}

\begin{proof}
Every nonzero proper ideal of $R$ is a radical ideal by Theorem~\ref{t1}. Thus every proper ideal of $R/nil(R)$ is a radical ideal, and hence $R/nil(R)$ is von Neumann regular.  The ``in particular"  and ``moreover'' statements are clear.
\end{proof}

\begin{exa} \label{e10}
{\rm (a) All proper ideals of $\mathbb{Z}_4$ are sdf-absorbing ideals, and all nonzero proper ideals (but not the zero ideal) of $\mathbb{Z}_{25}$ are sdf-absorbing ideals (cf. Theorem~\ref{t9}). Neither ring is reduced (i.e., von Neumann regular).

(b) All proper ideals of $\mathbb{Z}_2 \times \mathbb{Z}_2$ are sdf-absorbing ideals, and all nonzero proper ideals (but not the zero ideal) of $\mathbb{Z}_3 \times \mathbb{Z}_3$ are sdf-absorbing ideals (cf. Example~\ref{e100}). Both of these rings are von Neumann regular.

(c) However, not all nonzero proper ideals in a von Neumann regular ring need be sdf-absorbing ideals (cf. Example~\ref{e100}). For example, let $R = \mathbb{Z}_3 \times \mathbb{Z}_3 \times \mathbb{Z}_3$. Then the ideal $I = \{0\} \times \{0\} \times \mathbb{Z}_3$ is not an sdf-absorbing ideal of $R$ (to see this, let $a = (2,1,0), b = (1,1,0) \in R$). }
\end{exa}

The quasilocal case is easily handled.

\begin{thm} \label{t21}  Let $R$ be a quasilocal commutative ring with maximal ideal $M$. Then every nonzero proper ideal of $R$ is an sdf-absorbing ideal of $R$ if and only if $M$ is the unique prime ideal of $R$, $M$ is principal, and $M^2 = \{0\}$.
\end{thm}

\begin{proof}
We may assume that $R$ is not a field. Suppose that every nonzero proper ideal of $R$ is an sdf-absorbing ideal of $R$. Then $M$ is the unique prime ideal of $R$ by Theorem~\ref{t20}. Moreover, $M$ is the only nonzero proper ideal of $R$ since every nonzero proper ideal of $R$ is a radical ideal by Theorem~\ref{t1}. Thus $M$ is principal and $M^2 = \{0\}$.

Conversely, if $M$ is the unique prime ideal of $R$, $M$ is principal, and $M^2 = \{0\}$, then $M$ is the only nonzero proper ideal of $R$ and is an sdf-absorbing ideal of $R$.
\end{proof}

The next example shows that in the above theorem, $\{0\}$ may or may not be an sdf-absorbing ideal of $R$.

\begin{exa} \label{e50}
{\rm (a) Let $R = \mathbb{Z}_{p^2}$ for $p$ prime. Then $R$ satisfies the conditions of Theorem~\ref{t21}; so every nonzero proper ideal of $R$ is an sdf-absorbing ideal of $R$. However, $\{0\}$ is an sdf-absorbing ideal of $R$ if and only if $p = 2$ or $p = 3$ by Theorem~\ref{t9}.

(b)  Let $R = K[X]/(X^2)$, where $K$ is a field. Then $R$ satisfies the conditions of Theorem~\ref{t21}, so every nonzero proper ideal of $R$ is an sdf-absorbing ideal of $R$. However, it is easily verified that $\{0\}$ is an sdf-absorbing ideal of $R$ if and only if $K = \mathbb{Z}_3$.}
\end{exa}

The next several results consider the case where every proper ideal or nonzero proper ideal of a commutative von Neumann regular ring is an sdf-absorbing ideal. They depend on whether $2$ is zero, a unit, or a nonzero zero-divisor in $R$.

\begin{thm} \label{t5}  Let $R$ be a reduced commutative ring with $2 \in U(R)$.

{\rm (a)}  Every nonzero proper ideal of $R$ is an sdf-absorbing ideal of $R$ if and only if $R$ is a field or $R$ is isomorphic to $F_1 \times F_2$ for fields $F_1, F_2$.

{\rm (b)} Every proper ideal of $R$ is an sdf-absorbing ideal of $R$ if and only if $R$ is a field.
\end{thm}

\begin{proof}
{\rm (a)}  If $R$ is a field, then the claim is clear. So assume that $R$ is isomorphic to $F_1 \times F_2$ for fields $F_1, F_2$. Then $R$ has exactly two nonzero proper ideals and each is a maximal ideal of $R$. Thus every nonzero proper ideal of $R$ is an sdf-absorbing ideal of $R$.

Conversely, assume that $R$ is not a field and every nonzero proper ideal of $R$ is an sdf-absorbing ideal of $R$.   Suppose that $R$ has at least three distinct maximal ideals, say $M_1, M_2, M_3$. Then $M_1 \cap M_2 \neq  \{0\}$ (if $M_1 \cap M_2 = \{0\}$, then $M_1 \subseteq M_3$ or $M_2 \subseteq M_3$). Thus $M_1\cap M_2$ is a nonzero sdf-absorbing ideal of $R$ and $2 \in U(R)$; so $M_1 \cap M_2$ is a prime ideal of $R$ by Theorem \ref{t3}, a contradiction since dim$(R) =0$ by Theorem~\ref{t20}. Hence $R$ has exactly two maximal ideals, say $M_1$ and $M_2$. Then, arguing as above, $J(R) = M_1 \cap M_2 = \{0\}$; so $R$ is isomorphic to $F_1 \times F_2$ for fields $F_1 \cong R/M_1, F_2 \cong R/M_2$ by the Chinese Remainder Theorem.

{\rm (b)}  By part (a) above, we need only show that $I = \{(0, 0)\}$ is not an sdf-absorbing ideal of $R = F_1 \times F_2$ when $2 \in U(R)$. Let $a = (2, -2), b = (2, 2) \in F_1 \times F_2$. Then $a^2 - b^2 = (0, 0) \in I$, but $a + b = (4, 0) \not \in I$ and $a - b = (0, -4) \not \in I$ since char$(F_1)$, char$(F_2) \neq  2$ as $2 \in U(R)$. Thus $I = \{(0, 0)\}$ is not an sdf-absorbing ideal of $F_1 \times F_2$.
\end{proof}

Since every proper ideal of a commutative von Neumann regular ring is a radical ideal, Theorem \ref{t2} yields the following result which generalizes the fact that every ideal of a boolean ring is an sdf-absorbing ideal (Example~\ref{e1}(b)). Note that $R = \mathbb{F}_4 \times \mathbb{F}_4$ is a von Neumann regular ring with char$(R) = 2$, but $R$ is not a boolean ring.

\begin{thm}\label{t85}
Let $R$ be a commutative von Neumann regular ring with char($R) = 2$. Then every proper ideal of $R$ is an sdf-absorbing ideal of $R$.
\end{thm}

We next handle the case when $2$ is a nonzero zero-divisor of $R$.

\begin{thm}\label{t101}
Let $R$ be a commutative von Neumann regular ring with $ 0 \neq 2 \in Z(R)$. Then the following statements are equivalent.
	
{\rm (a)} Every proper ideal of $R$ is an sdf-absorbing ideal of $R$.

{\rm (b)}  Every nonzero proper ideal of $R$ is an sdf-absorbing ideal of $R$.

{\rm (c)} Exactly one maximal ideal $M$ of $R$ has char$(R/M) \neq 2$.
\end{thm}

\begin{proof}
$(a) \Rightarrow (b)$  This is clear.

$(b) \Rightarrow (c)$   First, assume that char$(R/M) = 2$ for every maximal ideal $M$ of $R$. Then $R$ is isomorphic to a subring of the direct product of fields of characteristic 2; so char$(R) = 2$, a contradiction.
Next, assume that $R$ has at least two maximal ideals $M_1, M_2$ with char$(R/M_1)$, char$(R/M_2) \neq 2$. Let $I = M_1 \cap M_2$. Then $I \ne \{0\}$ since otherwise $R$ is isomorphic to the direct product of two fields, each with characteristic $\neq 2$, by the Chinese Remainder Theorem.  Thus $I$ is an sdf-absorbing ideal of $R$, and hence $I/I = \{0\}$ is an sdf-absorbing ideal of $R/I$ by Corollary~\ref{c2}(b). However, $R/I$ is isomorphic to $F_1 \times F_2$ for fields $F_1 \cong R/M_1, F_2 \cong R/M_2$, where char$(F_1)$, char$(F_2) \neq 2$, by the Chinese Remainder Theorem, and thus $2 \in U(R/I)$. Hence $\{(0, 0)\}$ is not an sdf-absorbing ideal of $F_1 \times F_2$ by the proof of Theorem~\ref{t5}(b), a contradiction. Thus exactly one maximal ideal $M$ of $R$ has char$(R/M) \neq 2$.

$(c) \Rightarrow (a)$  Suppose that exactly one maximal ideal $M$ of $R$ has char$(R/M) \neq 2$. We need to show that every proper ideal $I$ of $R$ is an sdf-absorbing ideal of $R$.  Note that $I$ is a radical ideal of $R$ since $R$ is von Neumann regular. Thus it suffices to show that if $a^2 - b^2 \in I$, then either $a+b$ is in every prime (maximal) ideal of $R$ containing $I$ or $a-b$ is in every prime (maximal) ideal of $R$ containing $I$. Let $N$ be a maximal ideal of $R$ containing $I$; we consider two cases.

First, let $N = M$ be the unique maximal ideal $M$ of $R$ with char$(R/M) \neq 2$. If $I$ is contained in $M$, then as $a^2 - b^2\in M$, we have $a+b \in M$ or $a-b \in M$. Next, let $N \neq M$;
so char$(R/N)=2$. As $a^2 - b^2\in N$, both $a+b$ and $a-b$ are in $N$ by Theorem~\ref{l1}, and we select the sign convention to conform with the outcome of the previous case, if necessary.
So given the (radical) ideal $I$ of $R$, if $a^2-b^2\in I$, then either $a+b$ is in every prime (maximal) ideal of $R$ containing $I$ or $a-b$ is in every prime (maximal) ideal of $R$ containing $I$. Thus $a+b \in I$ or $a-b \in I$, and hence $I$ is an sdf-absorbing factor ideal of $R$.
\end{proof}

Together, Theorem~\ref{t5}, Theorem~\ref{t85}, and Theorem~\ref{t101} completely determine when every nonzero proper ideal or proper ideal of a commutative von Neumann regular ring is an sdf-absorbing ideal since every element in a commutative von Neumann regular ring is either a unit or a zero-divisor  (\cite[Corollary 2.4]{H}). We next apply these criteria to a commutative von Neumann regular ring which is the direct product of finitely many fields.

\begin{exa} \label{e100}
{\rm Let $R$ be the direct product of finitely many fields (so $R$ is von Neumann regular).

(a)  Every proper ideal of $R$ is an sdf-absorbing ideal of $R$ if and only if at most one of the fields has characteristic $\neq 2$.

(b)  Every nonzero proper ideal of $R$ is an sdf-absorbing ideal of $R$ if and only if at most one of the fields has characteristic $\neq 2$ or $R$ is the direct product of two fields.}
\end{exa}

\section{Additional results} \label{s4}

In this section, we give several more results about sdf-absorbing ideals in special classes of commutative rings. In particular, we determine the sdf-absorbing ideals in a PID and the direct product of two commutative rings, and study sdf-absorbing ideals in polynomial rings, idealizations, amalgamation rings, and $D + M$ constructions.

A nonzero sdf-absorbing ideal of a commutative ring $R$ is always a radical ideal of $R$ by Theorem~\ref{t1}; the following theorem gives a case where the converse holds.

\begin{thm} \label{t8}
Let $I$ be a proper ideal of a commutative ring $R$ such that $I = P_1\cap \cdots \cap P_n$ for distinct comaximal prime ideals $P_1, \ldots, P_n$ of $R$. Then $I$ is an sdf-absorbing ideal of $R$ if and only if at most one of the $P_i$'s has char$(R/P_i) \neq 2$.
\end{thm}

\begin{proof}  By way of contradiction, assume that $I$ is an sdf-absorbing ideal of $R$, $n \ge 2$, and char$(R/P_1)$, char$(R/P_2) \neq  2$. Then  $I/I = \{0\}$ is an sdf-absorbing ideal of $R/I$ by Corollary~\ref{c2}(b), and $R/I$ is isomorphic to  $T = R/P_1 \times \cdots \times R/P_n$ by the Chinese Remainder Theorem. Let $a = (1, 1,  \ldots, 1), b = (-1, 1, \dots, 1) \in T$. Then $a^2 - b^2 = (0, \ldots, 0)$, but $a + b = (0, 2, \ldots, 2) \not \in \{(0, \ldots , 0)\}$ and $a - b = (2, 0,\ldots, 0) \not \in \{(0,\ldots, 0)\}$. Thus $\{(0,\ldots, 0)\}$ is not an sdf-absorbing ideal of $T$, a contradiction. Hence at most one of the $P_i$'s has char$(R/P_i) \neq 2$.

The converse follows easily using a slight modification to the proof  of $(c) \Rightarrow (a)$ in Theorem~\ref{t101}. The details are left to the reader.
\end{proof}

\begin{remk} \label{r77}
{\rm
(a)  Theorem~\ref{t8} gives criteria for $nil(R)$  to be an sdf-absorbing ideal in a zero-dimensional semilocal commutative ring.

 (b)  In the proof of Theorem~\ref{t8}, note that if $I =  I_1\cap \cdots \cap I_n$ is an sdf-absorbing ideal of $R$ for distinct comaximal proper ideals $I_1, \ldots, I_n$ of $R$, then at most one of the $I_i$'s has char$(R/I_i) \neq 2$.}
\end{remk}

Using Theorem~\ref{t8}, we have the following characterization of sdf-absorbing ideals in a PID which extends Example~\ref{e1}(a) (also, cf. Example~\ref{e1}(f), Theorem~\ref{t2}, and Theorem~\ref{t3}).

\begin{cor}\label{c5}
Let $R$ be a PID and $I$ a nonzero proper ideal of $R$.

{\rm (a)} If $2$ is a nonzero nonunit of $R$, then $I$ is an sdf-absorbing ideal of $R$ if and only if $I$ is a prime {\rm (}maximal{\rm )} ideal of $R$,  i.e.,  $I = aR$ for $a\in R$ prime,  or $I = aR$, where $a = a_1 \cdots a_na_{n+1}$ for nonassociate primes $a_1, \ldots, a_{n + 1} \in R$ such that  $a_i \mid 2$ in $R$ for every $1 \leq i \leq n$.  In particular, if $2 \in R$ is prime, then $I$ is an sdf-absorbing ideal of $R$ if and only if $I$ is a prime {\rm (}maximal{\rm )} ideal of $R$ or $I = 2pR$ for  $p \in R$ prime  not associate to $2$.

{\rm (b)}  If $2 \in U(R)$, then $I$ is an sdf-absorbing ideal of $R$ if and only if $I$ is a prime {\rm (}maximal{\rm )} ideal of $R$, i.e.,  $I = aR$ for $a\in R$ prime.

{\rm (c)} If char$(R) = 2$, then $I$ is an sdf-absorbing ideal of $R$ if and only if $I$ is a radical ideal of $R$, i.e.,  $I = a_1 \cdots a_nR$ for nonassociate primes $a_1, \ldots, a_n \in R$.
\end{cor}

\begin{proof}
{\rm (a)} Assume that $I$ is a nonprime (and thus nonzero) sdf-absorbing ideal of the PID $R$. Then $I$ is a radical ideal of $R$ by Theorem~\ref{t1}, and hence $I$ is the intersection (product) of a finite number of distinct nonzero principal prime (maximal) ideals of $R$. Applying Theorem~\ref{t8}, we have $I = aR$, where $a = a_1\cdots a_na_{n+1}$ for nonassociate prime elements $a_1, \ldots , a_{n + 1} \in R$ such that  $a_i \mid 2$ (in $R$) for every $1 \leq i \leq n$.

The converse is clear by Theorem~\ref{t8}. The ``in particular'' statement is also clear.

{\rm (b)} This follows from Theorem~\ref{t3}.

{\rm (c)} This follows from Theorem~\ref{t1} and Theorem~\ref{t2}.
\end{proof}

We have the following example of a PID $R$ such that $2pR$ is not an sdf-absorbing ideal of $R$ for any prime  $p \in R$.

\begin{exa} \label{e80}
{\rm Let $R = \mathbb{Z}[i]$. Then $R$ is a PID and $2 = -i(1 + i)^2$, where $1 + i$ is prime and $i$ is a unit of $R$. Thus $2$ is not a prime element of $R$ and $2pR$ is not a radical ideal of $R$ for any prime $p \in R$; so $I = 2pR$ is not an sdf-absorbing ideal of $R$ for any prime $p \in R$ by Theorem~\ref{t1}. However, by Corollary~\ref{c5}(a), a nonprime ideal $I$ of $R$ is an sdf-absorbing ideal of $R$ if and only if $I = (1 +i)pR$ for some prime $p \in R$ not associate to $1 + i$.}
\end{exa}

We next consider when the two ideals $I[X]$ and $(I, X)$ are sdf-absorbing ideals of $R[X]$. The following partial result for $I[X]$ is a consequence of Theorem~\ref{t8}.

\begin{thm}\label{t205}
Let $I$ be a proper ideal of a commutative ring $R$ such that $I = P_1\cap \cdots \cap P_n$ for distinct comaximal prime ideals $P_1, \ldots, P_n$ of $R$. Then $I[X]$ is an sdf-absorbing ideal of $R[X]$ if and only if $I$ is an sdf-absorbing ideal of $R$.
\end{thm}

\begin{proof}
If $I[X]$ is an sdf-absorbing ideal of $R[X]$, then it is easily verified that $I$ is an sdf-absorbing ideal of $R$.

Conversely, assume that $I$ is an sdf-absorbing ideal of $R$. Then $I[X] = P_1[X]\cap \cdots \cap P_n[X]$, where  $P_1[X], \ldots , P_n[X]$ are distinct comaximal prime ideals of $R[X]$ since $P_1, \ldots, P_n$ are distinct comaximal prime ideals of $R$.  Moreover, at most one of the $P_i$'s has char$(R/P_i) \neq 2$ by Theorem~\ref{t8}. Since $R[X]/P_i[X]$ is isomorphic to $R/P_i[X]$ for every $1\leq i \leq n$, at most one of the $P_i[X]$'s has char$(R[X]/P_i[X]) \neq 2$. Thus $I[X]$ is an sdf-absorbing ideal of $R[X]$ by Theorem \ref{t8}.
\end{proof}

Combining Theorem \ref{t8} and Theorem \ref{t205}, we have the following corollary.

\begin{cor}\label{c309}
	Let $I$ be a proper ideal of a commutative ring $R$ such that $I = P_1\cap \cdots \cap P_n$ for distinct comaximal prime ideals $P_1, \ldots, P_n$ of $R$.  Then the following statements are equivalent.
	
{\rm (a)} $I$ is an sdf-absorbing ideal of $R$.

{\rm (b)} $I[X]$ is an sdf-absorbing ideal of $R[X]$.

{\rm (c)} At most one of the $P_i$'s has char$(R/P_i) \neq 2$.
\end{cor}

\begin{remk} \label{r111}
{\rm Note that $\{0\}$ may be an sdf-absorbing ideal in $R$, but not an sdf-absorbing ideal in $R[X]$ (so, in this case, $\{0\}$ is not the intersection of finitely many distinct comaximal prime ideals of $R$). For example, let $R = \mathbb{Z}_4$.}
\end{remk}

The sdf-absorbing ideals $(I, X)$ in $R[X]$ are easily classified
		
\begin{thm} \label{t569}
Let $R$ be a commutative ring.

{\rm (a)} Let $I$ be a nonzero proper ideal of $R$. Then $(I, X)$ is an sdf-absorbing ideal of $R[X]$ if and only if $I$ is an sdf-absorbing ideal of $R$.

{\rm (b)} $(X)$ is an sdf-absorbing ideal of $R[X]$ if and only if $R$ is reduced and $\{0\}$ is an sdf-absorbing ideal of $R$.
\end{thm}

\begin{proof}
(a) If $(I, X)$ is an sdf-absorbing ideal of $R[X]$, then $I$ is an sdf-absorbing ideal of $R$ by  Theorem~\ref{t7}(c).

Conversely, assume that $I$ is a nonzero sdf-absorbing ideal of $R$. Let $f = a + Xm(X), g = b + Xn(X) \in R[X]$ with $f^2 - g^2  \in (I, X)$. Then $a^2 - b^2 \in I$; so $a + b \in I$ or $a - b \in I$ by Remark~\ref{e0}(b). Thus $f + g \in (I, X)$ or $f - g \in (I, X)$; so $(I, X)$ is an sdf-absorbing ideal of $R[X]$.

(b) If $(X)$ is an sdf-absorbing ideal of $R[X]$, then $\{0\}$ is an sdf-absorbing ideal of $R$ by  Theorem~\ref{t7}(c). Moreover, $(X)$ is a radical ideal of $R[X]$ by Theorem~\ref{t1}; so $R$ is also reduced.

Conversely, assume that $R$ is reduced and $\{0\}$ is an sdf-absorbing ideal of $R$. Let $f = a + Xm(X), g = b + Xn(X) \in R[X]$ with $f^2 - g^2 \in (X)$. Then $a^2 - b^2 = 0$. If $a, b \neq 0$, then $a + b = 0$ or $a - b = 0$ since $\{0\}$ is an sdf-absorbing ideal of $R$. If $a  = 0$ or $b = 0$, then $a = b = 0$ since $R$ is reduced. So in either case, $f + g \in (X)$ or $f - g \in (X)$. Thus $(X)$ is an sdf-absorbing ideal of $R[X]$.
\end{proof}

The next theorem is similar to Theorem~\ref{t8}. Note that Corollary~\ref{c5} is also a consequence of Theorem~\ref{t2024}  since in Corollary~\ref{c5}, we have $P_1 \cap \cdots \cap P_n = P_1\cdots P_n$.

\begin{thm}\label{t2024}
Let $I$ be a proper ideal of a commutative ring $R$ such that $I = P_1\cap \cdots \cap P_n$ for prime ideals $P_1, \ldots, P_n$ of $R$ and the intersection of any $n-1$ of the ideals $P_1, \ldots, P_n$ is not equal to $I$. Then $I$ is an sdf-absorbing ideal of $R$ if and only if at most one of the $P_i$'s has char$(R/P_i) \neq 2$.
\end{thm}

\begin{proof}
Let $I$ be an sdf-absorbing ideal of $R$. By way of contradiction, assume that $n \ge 2$ and char$(R/P_1)$, char$(R/P_2) \neq  2$. Then $ I \subsetneq J = P_2 \cap \cdots \cap P_n$ by hypothesis; so  there is a $j \in J\setminus P_1$ and  $q \in P_1\setminus J$ (otherwise $I = P_1$). Let $x = j + q$ and $y = j - q$. Then  $x \neq 0$, $y \neq 0$, and $x^2 - y^2 = 4jq \in I$. Moreover,  $x + y = 2j \not \in P_1$ since $2 \not \in P_1$, and $x - y = 2q \not \in J$ since $2 \not \in P_2$. Thus $x + y \not \in I$ and  $x - y \not \in I$; so $I$ is not an sdf-absorbing ideal of $R$, a contradiction.  Hence at most one of the $P_i$'s has char$(R/P_i) \neq 2$.
	
The converse follows easily using a slight modification to the proof  of $(c) \Rightarrow (a)$ in Theorem~\ref{t101}. The details are left to the reader.
\end{proof}

In view of Theorem \ref{t2024}, we have the following example.

\begin{exa}
{\rm Let $R = \mathbb{Z}[X_1, \ldots, X_n]$ for $n \geq 2$. Then $I = (6, 2X_1, \ldots, 2X_n, X_1\cdots X_n) = (X_1, 3) \cap (X_2, 2) \cap \cdots \cap (X_n, 2)$ is an sdf-absorbing ideal of $R$ by Theorem~\ref{t2024}.}
\end{exa}

The next result completely determines when $\{0\}$ is an sdf-absorbing ideal of $\mathbb{Z}_n$.

\begin{thm} \label{t9}
$\{0\}$ is an sdf-absorbing ideal of $\mathbb{Z}_n$ if and only if  $n = 4, n = 9, n = p$ is prime, or $n = 2p$ for some odd prime $p$.
\end{thm}

\begin{proof}
Assume that $\{0\}$  is an sdf-absorbing ideal of $R = \mathbb{Z}_n$. First, suppose that $n \neq 2, 4$ is an even positive integer and $n \neq 2p$ for any odd prime $p$. We consider two cases. For the first case, assume that $4 \mid n$ in $\mathbb{Z}$. Then  the system of linear equations $X + Y = n/2$ and $X - Y = 2$ has a  solution $0 \neq x, y \in R$ and $x^2 - y^2 = 0$, but $x + y \neq 0$ and $x - y \neq 0$. Thus $\{0\}$ is not an sdf-absorbing ideal of $R$ by Theorem~\ref{t4}. For the second case, assume that $4  \nmid n$; so $n$ has an odd prime factor $q$. Then  the system of linear equations $X + Y = 2n/q$ and $X - Y = 2q$ has a solution $0 \neq x, y \in R$ (note that $n \neq 2p$ by assumption) and $x^2 - y^2 = 0$, but $x + y \neq 0$ and $x - y \neq 0$. Hence $\{0\}$ is not an sdf-absorbing ideal of $R$ by Theorem~\ref{t4}.  Next, suppose that $n$ is odd, not prime, $n \neq  3, 9$, and let $q$ be a prime factor of $n$. Then  the system of linear equations $X + Y = 4n/q$ and $X - Y = 2q$ has a  solution $0 \neq x, y \in R$ and $x^2 - y^2 = 0$, but  $x + y \neq 0$ and $x - y \neq 0$. Thus $\{0\}$ is not an sdf-absorbing ideal of $R$ by Theorem~\ref{t4} again.  Hence, if $\{0\}$ is an sdf-absorbing ideal of $\mathbb{Z}_n$, then $n = 4, n = 9, n = p$ is prime, or $n = 2p$ for some odd prime $p$.
	
Conversely, if $n = p$ is prime, then $\{0\}$ is a maximal ideal, and thus an sdf-absorbing ideal, of the field $\mathbb{Z}_n$. If $n = 4$ or $n = 9$, then one can easily verify that $\{0\}$ is an sdf-absorbing ideal of $\mathbb{Z}_n$. Finally, assume that $n = 2p$ for some odd prime $p$. Since $I = 2p\mathbb{\mathbb{Z}}$ is an sdf-absorbing ideal of $\mathbb{Z}$ by Example~\ref{e1}(a), we have that $I/I = \{0\}$ is an sdf-absorbing ideal of $\mathbb{Z}/2p\mathbb{Z} = \mathbb{Z}_{2p}$ by Corollary~\ref{c2}(b).
\end{proof}

We next investigate when the direct product of two ideals is an sdf-absorbing ideal. First, we consider the case when both ideals are nonzero proper ideals.

\begin{thm}\label{t600}
 Let $I_1, I_2$ be nonzero proper ideals of the commutative rings $R_1, R_2$, respectively. Then the following statements are equivalent.
	
{\rm (a)} $I_1 \times I_2$ is an sdf-absorbing ideal of $R_1 \times R_2$.

{\rm (b)} $I_1, I_2$ are sdf-absorbing ideals of $R_1, R_2$, respectively, and $2 \in I_1$ or $2 \in I_2$.
\end{thm}

\begin{proof}
$(a) \Rightarrow (b)$ Let  $I = I_1 \times I_2$ be an sdf-absorbing ideal of $R = R_1 \times R_2$. Then it is easily shown that $I_1, I_2$ are sdf-absorbing ideals of $R_1$, $R_2$, respectively. Next, let $a = (1, 1), b = (1, -1) \in R$. Then $a^2 - b^2 =  (0, 0) \in I$; so $(2,0) = a + b  \in I$ or $(0,2) = a - b) \in I$. Thus $2 \in I_1$ or $2 \in I_2$.
		
$(b) \Rightarrow (a)$  We may assume that $2 \in I_1$. Let $(0, 0) \not = a = (a_1, a_2), b = (b_1, b_2) \in R$ with $a^2 - b^2 \in I$. Then $a_1^2 - b_1^2 \in I_1$, and thus $a_1 + b_1 \in I_1$ or $a_1 - b_1 \in I_1$ by Remark~\ref{e0}(b) since $I_1$ is a nonzero sdf-absorbing ideal of $R_1$. Since $2 \in I_1$, we have $a_1 +b_1, a_1 - b_1 \in I_1$ by  Theorem~\ref{l1}. Also, $a_2^2 - b_2^2 \in I_2$; so $a_2 + b_2 \in I_2$ or $a_2 - b_2 \in I_2$ by Remark~\ref{e0}(b) again since $I_2$ is a nonzero sdf-absorbing ideal of $R_2$. If $a_2 + b_2 \in I_2$, then $a + b \in I$. If $a_2 - b_2 \in I_2$, then $a - b \in I$. Thus $I$ is an sdf-absorbing ideal of $R$.
\end{proof}

 Now we consider the case when one of the ideals in the product is either zero or the whole ring.

\begin{thm}\label{t600.5}
Let $I_1, I_2$ be nonzero proper ideals of the commutative rings $R_1, R_2$, respectively.

{\rm (a)} $\{0\} \times R_2$ is an sdf-absorbing ideal of $R_1 \times R_2$ if and only if $R_1$ is reduced and $\{0\}$ is an sdf-absorbing ideal of $R_1$. A similar result holds for $R_1 \times \{0\}$.
	
{\rm (b)} $I_1 \times R_2$ is an sdf-absorbing ideal of $R_1 \times R_2$ if and only if $I_1$ is an sdf-absorbing ideal of $R_1$. A similar result holds for $R_1 \times I_2$.

{\rm (c)} $\{0\} \times I_2$ is an sdf-absorbing ideal of $R_1 \times R_2$ if and only if $\{0\}$ is an sdf-absorbing ideal of $R_1$, $R_1$ is reduced, $I_2$ is an sdf-absorbing ideal of $R_2$, and char$(R_1) = 2$ or $2 \in I_2$. A similar result holds for $I_1 \times \{0\}$.
	
{\rm (d)} $\{0\} \times \{0\}$ is an sdf-absorbing ideal of $R_1 \times R_2$ if and only if $\{0\}$ is an sdf-absorbing ideal of $R_1$ and $R_2$, $R_ 1$ and $R_2$ are reduced, and char$(R_1) =2$ or char$(R_2) =2$.
\end{thm}

\begin{proof} Let $R = R_1 \times R_2$.

	{\rm (a)} Assume that  $\{0\} \times R_2$ is an sdf-absorbing ideal of $R$. Then it is clear that $\{0\}$ is an sdf-absorbing ideal of $R_1$. We show that $R_1$ is reduced. Let $c\in R_1$ with $c^2 = 0$, and $a = (c, 1), b = (0, 1) \in R$. Then $a^2 - b^2 = (0,0) \in \{0\} \times R_2$; so  $(c, 2) = a + b \in \{0\} \times R_2$ or $(c, 0) = a - b \in  \{0\} \times R_2$. Thus $c = 0$; so $R_1$ is reduced.
	
Conversely, assume that $R_1$ is reduced and $\{0\}$ is an sdf-absorbing ideal of $R_1$. Let $(0,0) \neq a = (a_1, a_2), b = (b_1, b_2) \in R$ with $a^2 - b^2 \in \{0\} \times R_2$. Since $R_1$ is reduced, we have  $a_1 = b_1 = 0$  or $0 \neq a_1, b_1 \in R_1$. Hence $a_1 + b_1 = 0$ or $a_1 - b_1 = 0$; so $a + b \in  \{0\} \times R_2$  or $ a - b \in \{0\} \times R_2$. Thus $\{0\} \times R_2$ is an sdf-absorbing ideal of $R$.

{\rm (b)}  The proof is similar to that of Theorem~\ref{t600}.
	
{\rm (c)} The proof is similar to part (a) above.
	
{\rm (d)} Let $I = \{0\} \times \{0\}$ be an sdf-absorbing ideal of $R$. Then it is easily shown that $\{0\}$ is an sdf-absorbing ideal of $R_1$ and $R_2$. By an argument similar to that in part (a) above, we have that $R_1$ and $R_2$ are reduced.  Let $a = (1, 1), b = (1, -1) \in R$. Then $a^2 - b^2 =  (0, 0) \in I$; so $(2,0) = a + b  \in I$ or $(0,2) = a - b  \in I$. Thus char$(R_1) = 2$ or char$(R_2)  = 2$.

  Conversely, assume that $\{0\}$ is an sdf-absorbing ideal of $R_1$ and $R_2$, $R_ 1$ and $R_2$ are reduced, and char$(R_1) =2$ or char$(R_2) =2$. We may assume that char$(R_2) =2$. Let $(0, 0) \neq  a = (a_1, a_2), b = (b_1, b_2) \in R$ with $a^2 - b^2 = (0, 0)$. Since $R_1$ is reduced, we have $a_1= b_1  = 0$ or $0 \neq a_1, b_1  \in R_1$. Since  $a_1^2 - b_1^2 = 0$, we have  $a_1 + b_1 = 0$ or  $a_1 - b_1 = 0$. Similarly, $a_2 + b_2 = a_2 - b_2 = 0$ since char$(R_2) = 2$. If $a_1 + b_1 = 0$, then $a + b = (0, 0)$. If $a_1 - b_1 = 0$, then $a - b = (0, 0)$. Thus $I = \{(0, 0)\}$ is an sdf-absorbing ideal of $R$.
\end{proof}

\begin{remk} \label{r815}
{\rm (a) The previous two theorems may be combined by an abuse of definition (consider the whole ring to be an sdf-absorbing radical ideal, and note that $2 \in \{0\}$ if and only if the ring has characteristic $2$). Let $I_1, I_2$ be ideals of $R_1, R_2$, respectively, not both the whole ring.  Then $I_1 \times I_2$ is an sdf-absorbing ideal of $R_1 \times R_2$ if and only if $I_1, I_2$ are sdf-absorbing radical ideals of $R_1, R_2$, respectively, and char$(R_1) = 2$ or char$(R_2) = 2$.

(b) $\{0\}$ and $\{0, 2\}$ are sdf-absorbing ideals of $\mathbb{Z}_4$, but $\{0\} \times \{0\}$, $\{0\} \times \{0,2\}$, and $\{0\} \times \mathbb{Z}_4$ are not sdf-absorbing ideals of $\mathbb{Z}_4 \times \mathbb{Z}_4$ by Theorem~\ref{t600.5} (or choose $a = (2,1), b= (0,1)$). Also, see Example~\ref{e707}.}
\end{remk}

In view of Example~\ref{e1}(a), Theorem~\ref{t600}, and Theorem~\ref{t600.5}, we have the following example.

\begin{exa} \label{e707}
{\rm Let $R = \mathbb{Z} \times \mathbb{Z}$ and $p \in \mathbb{Z}$ a positive prime. Then a nonzero ideal $I$ of $R$ is an sdf-absorbing ideal of $R$ if and only if $I$ is a prime ideal of $R$ (i.e., $I = \{0\} \times \mathbb{Z}$, $I = p\mathbb{Z} \times \mathbb{Z}$, $I = \mathbb{Z} \times \{0\}$, or $I = \mathbb{Z} \times p\mathbb{Z}$), $I = 2\mathbb{Z} \times p\mathbb{Z}$,  $I = p\mathbb{Z} \times 2\mathbb{Z}$, $I = 2p\mathbb{Z} \times \mathbb{Z}$ ($p \not = 2$), $I = \mathbb{Z} \times 2p\mathbb{Z}$ ($p \neq 2$), $I = 2\mathbb{Z} \times 2p\mathbb{Z}$ ($p \neq 2$),  $I = 2p\mathbb{Z} \times 2\mathbb{Z}$ ($p \neq 2$),  $I  = \{0\} \times 2\mathbb{Z}$, or $I = 2\mathbb{Z} \times \{0\}$.

The ideals $\{0\} \times \{0\}, \{0\} \times p\mathbb{Z}$ ($p \neq 2$),  $p\mathbb{Z} \times \{0\}$  ($p \neq 2$),  $\{0\} \times 2p\mathbb{Z}$ ($p \neq 2$), and $2p\mathbb{Z} \times \{0\}$  ($p \neq 2$) are not sdf-absorbing ideals of $R$ (or choose $a = (1,1), b = (1,-1)$).}
\end{exa}

In the following result, we determine the sdf-absorbing ideals in idealization rings. Recall that for a commutative ring $R$ and $R$-module $M$, the {\it idealization of $R$ and $M$} is the commutative ring $R(+)M = R \times M$ with identity $(1,0)$ under addition defined by $(r,m) + (s,n) = (r+s,m + n)$ and multiplication defined by $(r,m)(s,n) = (rs,rn + sm)$. For more on idealizations, see \cite{AW, H}.  Every ideal of $R(+)M$ has the form $I(+)N$ for $I$ an ideal of $R$ and $N$ a submodule of $M$ (\cite[Theorem 25.1(1)]{H}); so Theorem~\ref{t907} completely determines the nonzero sdf-absorbing ideals of $R(+)M$.

\begin{thm}\label{t907}
Let $R$ be a commutative ring, $I$ a nonzero proper ideal of $R$, $M$  an $R$-module, and $N$ a submodule of $M$. Then $I (+) N$ is an sdf-absorbing ideal of $R (+) M$ if and only if  $I$ is an sdf-absorbing ideal of $R$ and $N = M$.
\end{thm}

\begin{proof} Let $A = R (+) M$ and $J = I (+) N$.

Assume that $J$ is an sdf-absorbing ideal of $A$. It is easily verified that $I$ is an sdf-absorbing ideal of $R$. By way of contradiction, assume that $N \subsetneq M$; so there is an $m \in M \setminus N$. Let $0 \neq i \in I$, and $a = (i, 0), b = (0, m) \in R(+)M = A$. Then $a^2 - b^2 = (i^2,  0) \in I(+)N = J$, but $a + b = (i, m) \not \in J$ and $a - b = (i, -m) \not \in J$, a contradiction. Thus $N = M$.

Conversely, assume that $I$ is an sdf-absorbing ideal of $R$ and $N = M$. Let $a^2 - b^2 \in J$  for $(0, 0) \neq a, b \in A$, where $a = (a_1, m_1)$ and $b =(b_1, m_2)$. Since $I$  is a nonzero sdf-absorbing ideal of $R$ and $a_1^2 - b_1^2 \in I$, we have  $a_1 + b_1 \in I$ or  $a_1 - b_1 \in I$ by Remark~\ref{e0}(b). If $a_1 + b_1 \in I$, then $a + b \in  I (+) M = J$. If $a_1 - b_1 \in I$, then $a - b \in  I (+) M = J$. Thus $J =  I (+) M$ is an sdf-absorbing ideal of $A$.
\end{proof}

The following example shows that it is crucial that $I$ be a nonzero ideal in Theorem~\ref{t907}.

\begin{exa} \label{e101}
{\rm  Let $R = \mathbb{Z}_4$, $M =N = \mathbb{Z}_4$, and $I = \{0\}$. Then $\{0\}$ is an sdf-absorbing ideal of $\mathbb{Z}_4$, but $\{0\} (+) \mathbb{Z}_4$ is not an sdf-absorbing ideal of $\mathbb{Z}_4 (+) \mathbb{Z}_4$ by Theorem~\ref{t1} since  $\{0\} (+) \mathbb{Z}_4$ is not a radical ideal of $\mathbb{Z}_4 (+) \mathbb{Z}_4$  (or consider $x = (2,0), y = (0,2)$).}
\end{exa}

Next, we consider when $\{0\}(+)N$ is an sdf-absorbing ideal of $R(+)M$.

\begin{remk} \label{r101}
{ \rm Let $R$ be a commutative ring and $M$ a nonzero $R$-module.

 (a) If $\{0\}(+)N$ is an sdf-absorbing ideal of $R(+)M$ for $N$ a proper submodule of $M$, then $N = \{0\}$ by Theorem~\ref{t1}.

(b) It is easily shown that $\{0\}(+)M$ is an sdf-absorbing ideal of $R(+)M$ if and only if $R$ is reduced and $\{0\}$ is an sdf-absorbing ideal of $R$ (cf. Example~\ref{e101}).

(c) It is easily shown that $\{(0,0)\}$ is not an sdf-absorbing ideal of $R(+)M$ when $|M| \neq 3$. However, $\{(0,0)\}$ is an sdf-absorbing ideal of $\mathbb{Z}_3(+)\mathbb{Z}_3$, but not of $\mathbb{Z}(+)\mathbb{Z}_3$.}
\end{remk}
	
In the next result, we study sdf-absorbing ideals in amalgamation rings. Let $A, B$ be commutative rings, $f: A \longrightarrow B$ a homomorphism, and $J$ an ideal of $B$. Recall that the {\it amalgamation of $A$ and $B$ with respect to $f$ along $J$} is the subring $A \bowtie_J B = \{(a, f(a) + j ) \mid a \in A, j \in J\}$ of $A \times B$.

\begin{thm}\label{t709}
Let $A$ and $B$ be commutative rings, $f: A \longrightarrow B$ a homomorphism, $J$ an ideal of $B$, and $I$ a nonzero proper ideal of $A$. Then $I \bowtie_J B$ is an sdf-absorbing ideal of $A\bowtie_J B$ if and only if $I$ is an sdf-absorbing ideal of $A$.
\end{thm}

\begin{proof}
If $I \bowtie_J B$ is an sdf-absorbing ideal of $A\bowtie_J B$, then it is easily verified $I$ is an sdf-absorbing ideal of $A$.

Conversely,  assume that $I$ is a nonzero sdf-absorbing ideal of $A$. Let $x = (a, f(a)  + j_1), y = (b, f(b) + j_2) \in A\bowtie_J B$ such that $x^2 - y^2 \in I \bowtie_J B$. Since $a^2 - b^2 \in I$ and $I$ is a nonzero sdf-absorbing ideal of $A$, we have $a + b \in I$ or $a - b \in I$ by Remark~\ref{e0}(b). If $a + b\in I$, then  $x + y = (a + b, f(a) + j_1 + f(b) + j_2) = (a + b, f(a + b) + j_1 +j_2)\in I\bowtie_J B$. Similarly, if $a - b \in I$, then $x - y = (a - b, f(a -b) + j_1 - j_2) \in I\bowtie_J B$. Thus $I\bowtie_J B$ is an sdf-absorbing ideal of $A\bowtie_J B$.
\end{proof}

The following example shows that it is again crucial that $I$ be a nonzero ideal in Theorem~\ref{t709}.

\begin{exa}
{\rm Let $A = B = J = \mathbb{Z}_4$, $f = 1_A : A \longrightarrow A$, and $I = \{0\}$. Then $\{0\}$ is an sdf-absorbing ideal of $\mathbb{Z}_4$, but $\{0\} \bowtie_{\mathbb{Z}_4} \mathbb{Z}_4$ is not an sdf-absorbing ideal of $\mathbb{Z}_4\bowtie_{\mathbb{Z}_4} \mathbb{Z}_4$ by Theorem~\ref{t1} since  $\{0\}\bowtie_{\mathbb{Z}_4} \mathbb{Z}_4 \neq \{(0, 0)\}$ is not a radical ideal of $\mathbb{Z}_4 \bowtie_{\mathbb{Z}_4} \mathbb{Z}_4$ (or consider $x = (2,0), y = (0,2)$).}
\end{exa}

Let $T$ be an integral domain of the form $K + M$, where the field $K$ is a subring of $T$ and $M$ is a nonzero maximal ideal of $T$, and let $D$ be a subring of $K$. Then $R = D + M$ is a subring of $T$ with the same quotient field as $T$. This ``$D + M$'' construction has proved very useful for constructing examples since ring-theoretic properties of $R$ are often determined by those of $T$ and $D$. The ``classical'' case, when $T$ is a valuation domain, was first studied systemically in \cite[Appendix II]{GQ}, and the ``generalized'' $D + M$ construction as above was introduced and studied in \cite{BR}. The next several results concern the sdf-absorbing ideals in $D + M$. (The relevant facts concerning ideals in $D + M$ used in the proof of Theorem~\ref{t763} may be found in \cite[Theorem A, p. 560]{GQ} or \cite[Exercise 11, p. 202]{G}.)

\begin{thm} \label{t763}  Let $T = K + M$ be an integral domain, where the field $K$ is a subring of $T$ and $M$ is a nonzero  maximal ideal of $T$, and let $D$ be a subring of $K$ and $R = D + M$.

{\rm (a)}  Let $I$ be an ideal of $D$. Then $I + M$ is an sdf-absorbing ideal of $R$ if and only if $I$ is an sdf-absorbing ideal of $D$.

{\rm (b)} Let $T$ be a valuation domain. Then $J$ is an sdf-absorbing ideal of $R$ if and only if $J = I +M$, where $I$ is an sdf-absorbing ideal of $D$, or $J$ is a prime ideal of $T$.
\end{thm}

\begin{proof}
(a) This follows directly from Theorem~\ref{t7}(a)(c).

(b) Let $J$ be an ideal of $R$; so $J$ is comparable to $M$. If $M \subseteq J$, then $J = I + M$ for an ideal $I$ of $D$. Thus $J$ is an sdf-absorbing ideal of $R$ if and only if $I$ is an sdf-absorbing ideal of $D$ by part (a) above. So we may assume that $J \subseteq M$. Note that the prime ideals of $R$ contained in $M$ are just the prime ideals of $T$. Hence the radical ideals of $R$ contained in $M$ are precisely the prime ideals of $T$ (since radical ideals in a valuation domain are prime). Thus $J$ is an sdf-absorbing ideal of $R$ if and only if $J$ is a prime ideal of $T$.
\end{proof}

\begin{exa}
{\rm Let $T = \mathbb{Q}[[X]] = \mathbb{Q} + X\mathbb{Q}[[X]]$, $R = \mathbb{Z} + X\mathbb{Q}[[X]]$, and $p \in \mathbb{Z}$ a positive prime. Then $T$ is a valuation domain (DVR) with maximal ideal $M = X\mathbb{Q}[[X]]$. Thus the sdf-absorbing ideals of $R$ are the prime ideals $\{0\}$, $X\mathbb{Q}[[X]]$, and $p\mathbb{Z} + X\mathbb{Q}[[X]]$, and the ideals $2p\mathbb{Z} + X\mathbb{Q}[[X]]$ ($p \neq 2$) by Theorem~\ref{t763}(b) and Example~\ref{e1}(a). Note that $R$ is a B\'{e}zout domain by \cite[Theorem 7]{BR}, but not a PID since $M$ is not a principal (or even finitely generated) ideal of $R$.}
\end{exa}

\section{Weakly square-difference factor absorbing ideals} \label{s5}

Recall from \cite{DS} (also see \cite{bb}) that a proper ideal $I$ of a commutative ring $R$ is a {\it weakly prime ideal} of $R$ if whenever $0\not = ab \in I$ for  $a, b \in R$, then $a\in I$ or $b\in I$. In this section, we introduce and study the ``weakly'' analog of sdf-absorbing ideals. First we give the definition.

\begin{definition}
{\rm A proper ideal $I$ of a commutative ring $R$ is a {\it weakly square-difference factor absorbing ideal} (weakly sdf-absorbing ideal) of $R$ if whenever  $0 \not = a^2 - b^2 \in I$ for $0 \neq a, b \in R$, then $a + b \in I$ or  $a - b \in I$.}
\end{definition}

A weakly prime ideal or sdf-absorbing ideal of $R$ is clearly also a weakly sdf-absorbing ideal of $R$. If $R$ is an integral domain, then $I$ is a weakly prime (resp., weakly sdf-absorbing) ideal of $R$ if and only if it is a prime (resp., sdf-absorbing) ideal of $R$. Also, $\{0\}$ is vacuously a weakly prime and weakly sdf-absorbing ideal of $R$, but need not be a prime or sdf-absorbing ideal of $R$. The following is an example of a nonzero weakly sdf-absorbing ideal that is neither an sdf-absorbing ideal nor a weakly prime ideal.

\begin{exa}\label{ew1}
{\rm Let  $R = \mathbb{Z}_4 \times \mathbb{Z}_4$, and $I = \{0\} \times \{0,2\}$. Then $I$ is not a radical ideal of $R$; so $I$  is not an sdf-absorbing ideal of $R$ by Theorem~\ref{t1}. Also, $(0, 0) \not = (2, 2)(0, 1) \in I$, but  $(2, 2) \not \in I$ and $(0, 1)\not \in I$; so $I$ is not a weakly prime ideal of $R$. Note that if  $x^2 - y^2 \in I$ for $x, y \in R$, then $x^2 - y^2 = (0, 0)$.   Thus $I$ is a weakly sdf-absorbing ideal of $R$.}
\end{exa}

The next theorem is the ``weakly''  version of Theorem~\ref{t3}.

\begin{thm}\label{tw1}
Let $I$ be a weakly sdf-absorbing ideal of a commutative ring $R$ with $2 \in U(R)$. Then $I$ is a weakly prime ideal of $R$.
\end{thm}

\begin{proof}
The proof is similar to the proof of  Theorem~\ref{t3}. The details are left to the reader.
\end{proof}

The next two theorems and corollary are the ``weakly'' analogs of Theorem ~\ref{t349}, Theorem~\ref{t7}(b)(c), and Corollary~\ref{c2}(a)(b), respectively. They follow directly from the definitions; so their proofs are omitted.

\begin{thm} \label{te50}
Let $I$ be a weakly sdf-absorbing ideal of a commutative ring $R$, and let $S$ be a multiplicatively closed subset of $R$ with $I \cap S = \emptyset$. Then $I_S$ is a weakly sdf-absorbing ideal of $R_S$.
\end{thm}

\begin{thm} \label{te70}
Let $f : R \longrightarrow T$ be a homomorphism of commutative rings.

{\rm (a)} If $f$ is injective and $J$ is a weakly sdf-absorbing ideal of $T$, then $f^{-1}(J)$ is a weakly sdf-absorbing ideal of $R$.

{\rm (b)} If $f$  is surjective and $I$ is a weakly sdf-absorbing ideal of $R$ containing $ker(f)$, then  $f(I)$ is a weakly sdf-absorbing ideal of $T$.
\end{thm}

\begin{cor} \label{ce50}
{\rm (a)} Let $R \subseteq T$ be an extension of commutative rings and $J$ a weakly sdf-absorbing ideal of $T$. Then $J \cap R$ is a weakly sdf-absorbing ideal of $R$.

{\rm (b)}  Let $ J \subseteq I$ be ideals of a commutative ring R. If $I$ is a weakly sdf-absorbing ideal of $R$, then $I/J$ is a weakly sdf-absorbing ideal of $R/J$.

\end{cor}

The following examples (cf. Example~\ref{e2}) show that the ``weakly'' analog of Theorem~\ref{t7}(a) and Corollary~\ref{c2}(c) may fail and the ``$ker(f) \subseteq I$'' hypothesis is needed in Theorem~\ref{te70}(b).
	
\begin{exa}\label{e200}
{\rm (a) Let $f: \mathbb{Z} \times \mathbb{Z} \longrightarrow \mathbb{Z}/4\mathbb{Z} \times \mathbb{Z}/4\mathbb{Z} = \mathbb{Z}_4 \times \mathbb{Z}_4$ be the natural epimorphism. By Example~\ref{ew1},  $\{0\} \times \{0,2\}$ is a weakly sdf-absorbing ideal of $\mathbb{Z}_4 \times \mathbb{Z}_4$, but $f^{-1}(\{0\} \times \{0,2\}) =  4\mathbb{Z} \times 2\mathbb{Z}$ is not a weakly sdf-absorbing ideal of $\mathbb{Z} \times \mathbb{Z}$ (let $a = (2,2)$ and $b = (0,2)$). Thus the ``weakly'' analog of Theorem~\ref{t7}(a) and Corollary~\ref{c2}(c) may fail.

(b)  Let $f : \mathbb{Z}[X] \longrightarrow \mathbb{Z}$ be the epimorphism given by $f(g(X)) = g(0)$. Then $I = (X + 4)$ is a prime ideal, and thus a weakly sdf-absorbing ideal, of $\mathbb{Z}[X]$, but  $f((X+4)) = 4\mathbb{Z}$ is not a weakly sdf-absorbing ideal of $\mathbb{Z}$ (let $a = 4$ and $b = 2$). Note that $ker(f) = (X) \not \subseteq (X + 4) = I$; so the ``$ker(f) \subseteq I$'' hypothesis is needed in Theorem~\ref{te70}(b).}
\end{exa}

If $I$ is a weakly prime ideal of a commutative ring $R$ that is not a prime ideal, then $I \subseteq nil(R)$ by  \cite[Theorem 1]{DS}.  A similar result holds for weakly sdf-absorbing ideals.

\begin{thm}\label{tw2}
Let $I$ be a weakly sdf-absorbing ideal of a commutative ring $R$. If $I$ is not an sdf-absorbing ideal of $R$, then $I \subseteq nil(R)$.
\end{thm}

\begin{proof}
Since $I$ is not an sdf-absorbing ideal of $R$, we have $a^2 - b^2 = 0$ for some $0 \neq a, b \in R$, but $a + b \not \in I$ and $a - b \not \in I$. Note that if $a, b \in I$, then $a + b \in I$ and $a - b \in I$, a contradiction. So without loss of generality, we may assume that $b \not \in I$. Let $i \in I$. Then $b + i, b - i \neq 0$. We first show that $a^2 - (b + i)^2 = 0$  and $a^2 - (b - i)^2 = 0$.  Since $i \in I$ and $a^2 - b^2 = 0$, we have $a^2 - (b + i)^2 = a^2 -  b^2 - 2bi - i^2 = -2bi -i^2\in I$. Suppose that $a^2 - (b + i)^2 \not = 0$. Since $I$ is a weakly sdf-absorbing ideal of $R$, either $a + (b + i) \in I$ or $a - (b + i) \in I$.  Thus $a + b \in I$ or $a - b \in I$,  a contradiction. Similarly, $a^2 - (b - i)^2 = 0$. Hence $-2bi - i^2 = a^2 -  b^2 - 2bi - i^2 =  a^2 - (b + i)^2= 0$ and  $2bi - i^2 =  a^2 -  b^2  + 2bi - i^2  = a^2 - (b - i)^2 = 0$; so $2i^2 = 0$.  Thus $2i \in nil(R)$. Since $2bi - i^2 = 0$  and $2i \in nil(R)$, we have $i^2 = 2bi \in nil(R)$. Hence $i \in nil(R)$, and thus $I \subseteq nil(R)$.
\end{proof}

In light of the proof of Theorem \ref{tw2}, we have the following result.

\begin{cor}\label{cw2}
Let $I$ be a weakly sdf-absorbing ideal of a commutative ring $R$ that is not an sdf-absorbing ideal of $R$.

{\rm (a)}  $2i^2 = 0$, and hence $2i \in nil(R)$, for every $i \in I$. Moreover, if $2 \not \in Z(R)$ or char$(R) = 2$, then $i^2 = 0$ for every $i \in I$.

{\rm (b)} If $R$ is reduced, then $I = \{0\}$.
\end{cor}

We next investigate when $I \times J$ is a weakly sdf-absorbing ideal of $R_1 \times R_2$.

\begin{thm}
Let $R_1, R_2$ be commutative rings and $I$ a nonzero weakly sdf-absorbing ideal of $R_1$. Then the following statements are equivalent.
	
{\rm (a)} $ I \times R_2$ is a weakly sdf-absorbing ideal of $R_1 \times R_2$.

{\rm (b)} $I$ is an sdf-absorbing ideal of $R_1$.

{\rm (c)} $I \times R_2$ is an sdf-absorbing ideal of $R_1 \times R_2$.
\end{thm}

\begin{proof}
	$(a) \Rightarrow (b)$ Let $R = R_1 \times R_2$ and $J = I \times R_2$. Assume by way of contradiction that $I$ is not an sdf-absorbing ideal of $R_1$. Since $I$ is a weakly sdf-absorbing ideal of $R_1$, there are $0 \neq a, b \in R_1$ such that $a^2 - b^2 = 0$, but $a + b \not \in I$ and $a - b \not \in I$. Let $x = (a, 1), y = (b, 0) \in R$. Then $0 \neq x, y \in R$ and $(0, 0) \not = x^2 - y^2 \in J$. Since $J$ is a weakly sdf-absorbing ideal of $R$, we have   $x + y = (a + b, 1) \in J$ or $x - y = (a - b, 1) \in J$. Thus $a + b \in I$ or $a - b \in I$, a contradiction. Hence $I$ is an sdf-absorbing ideal of $R_1$.
	
	$(b) \Rightarrow (c)$ This follows from Remark~\ref{r815}(b).
	
	$(c) \Rightarrow (a)$ This is clear.
	\end{proof}

\begin{thm} \label{weakly300}
Let $R_1, R_2$ be commutative rings, $I$ a weakly sdf-absorbing ideal of $R_1$ that is not an sdf-absorbing ideal, and $J$  a weakly sdf-absorbing ideal of $R_2$ that is not an sdf-absorbing ideal. Then the following statements are equivalent
	
{\rm (a)} $I \times J$ is a weakly sdf-absorbing ideal of $R_1 \times R_2$ that is not an sdf-absorbing ideal.
		
{\rm (b)} $I \times J$ is a weakly sdf-absorbing ideal of $R_1 \times R_2$.

{\rm (c)} If $a^2 - b^2 \in I$ for $a, b \in R_1$, then $a^2 - b^2 = 0$, and if $c^2 - d^2 \in J$ for  $c, d \in R_2$, then $c^2 - d^2 = 0$.

{\rm (d)} If $x^2 - y^2 \in I \times J$ for $0 \neq x, y \in R_1 \times R_2$, then $x^2 - y^2 = (0, 0)$.
			
\end{thm}

\begin{proof}  Let $R  = R_1 \times R_2$ and $K = I \times J$.

$(a)\Rightarrow (b)$ This is clear.

$(b) \Rightarrow (c)$ Assume $ 0 \neq  a^2 - b^2 \in I$ for $a, b \in R_1$. Since $J$ is a weakly sdf-absorbing ideal of $R_2$ that is not an sdf-absorbing ideal, we have $e^2 - f^2 = 0$ for some $0 \neq e, f \in R_2$, but $e + f \not \in J$ and $e - f \not \in J$. Let $x = (a, e)$ and $y = (b, f)$. Then $0 \neq x, y \in R$ and $0 \neq x^2 - y^2 = (a^2 - b^2, e^2 - f^2) \in K$. Since $K$ is a weakly sdf-absorbing ideal of $R$, we have $x + y = (a + b, e + f) \in K$ or $x - y = (a - b, e - f) \in K$. Thus $e + f \in J$ or $e - f \in J$,  a contradiction. Hence if $a^2 - b^2 \in I$ for  $a, b \in R_1$, then $a^2 - b^2 = 0$. A similar argument shows that if $c^2 - d^2 \in J$ for $c, d \in R_2$, then $c^2 - d^2 = 0$.
	
$(c)\Rightarrow (d)$ This is clear.
	
$(d) \Rightarrow (a)$ Clearly $K$ is a weakly sdf-absorbing ideal of $R$. We show that $K$ is not a $1$-absorbing ideal of $R$. Since $I$ is a weakly sdf-absorbing ideal of $R_1$ that is not an sdf-absorbing ideal, we have $a^2 - b^2 = 0$ for some $0 \neq a,b \in R_1$, but $a + b \not \in I$ and $a -b \not \in I$. Let $x = (a,0), y = (0,b)$. Then $0 \neq x,y \in R$ and $x^2 - y^2 = (0,0) \in K$, but $x + y = (a + b,0) \not \in K$ and $x - y = (a - b,0) \not \in K$.  Hence $K$ is not an sdf-absorbing ideal of $R$.
\end{proof}

Note that in the proof of $(d) \Rightarrow (a)$ of Theorem~\ref{weakly300} above, we only need one of $I, J$ to not be an sdf-absorbing ideal. In view of Theorem~\ref{weakly300}, the following example shows that $I\times J$ may be a weakly sdf-absorbing ideal of $R_1 \times R_2$ that is not an sdf-absorbing ideal, but neither $I$ nor $J$ need be a weakly sdf-absorbing ideal that is not an sdf-absorbing ideal.

\begin{exa} \label{ew2}
{\rm Let $R_1 = R_2 = \mathbb{Z}_4$, $R = R_1 \times R_2$, and $K = \{0\} \times \{0, 2\}$. Then $K$ is a nonzero weakly sdf-absorbing ideal of $R$ that is not an sdf-absorbing ideal by Example \ref{ew1}. However, $I = \{0\}$ is an sdf-absorbing ideal of $R_1$ and $J = \{0, 2\}$ is an sdf-absorbing ideal of $R_2$.}
\end{exa}

The following example satisfies the hypothesis of Theorem~\ref{weakly300}.

\begin{exa}\label{ew3}
{\rm	Let $R_1 = \mathbb{Z}_2[X]/(X^2)$, $R_2 = \mathbb{Z}_4 \times \mathbb{Z}_4$, and $R = R_1\times R_2$. Then $I = \{0\}$ is a weakly sdf-absorbing ideal of $R_1$. Since $(X+1)^2 - 1^2 = 0$ in $R_1$, but  $X \not \in I$, we have that $I$ is not an sdf-absorbing ideal of $R_1$. Let $J= \{0\} \times \{0, 2\}$. Then $J$ is a weakly sdf-absorbing ideal of $R_2$ that is not an sdf-absorbing ideal by Example~\ref{ew1}. Since $x^2 - y^2 \in K = I\times J$ for  $0 \neq x, y \in R$ implies $x^2 - y^2 = (0, 0, 0) \in R$, we have  that $K = I \times J$ is a weakly sdf-absorbing ideal of $R$ that is not an sdf-absorbing ideal by Theorem~\ref{weakly300}.}
\end{exa}


 In Theorem~\ref{t9}, we determined when $\{0\}$ is an sdf-absorbing ideal of $\mathbb{Z}_n$. We next consider $nil(\mathbb{Z}_n)$ ($= J(\mathbb{Z}_n)$).

\begin{thm}
{\rm (a)} $nil(\mathbb{Z}_n)$ is an sdf-absorbing ideal of $\mathbb{Z}_n$  if and only if $n = q^m$ for some integer $m \geq 1$ and  positive prime $q$ or $n = 2^ip^k$ for some integers $i, k \geq 1$ and positive prime $p \neq  2$.

{\rm (b)} $nil(\mathbb{Z}_n) =\{0\}$ is a weakly sdf-absorbing ideal of $\mathbb{Z}_n$ that is not an sdf-absorbing ideal of $\mathbb{Z}_n$ if and only if $n = pq$ for distinct odd  positive primes $p, q$ or  $n = p_1 \cdots p_m$ for distinct positive primes $p_1, \ldots, p_m$, where $m \geq 3$.
\end{thm}

\begin{proof} {\rm (a)} This follows from Theorem~\ref{t2024}.

{\rm (b)} Note that $nil(\mathbb{Z}_n) = \{0\}$ if and only if $n$ is a product of distinct positive primes. The result then follows from Theorem~\ref{t2024} and Theorem~\ref{t9}.
\end{proof}

\end{document}